\documentclass{amsart}

\usepackage{thesis} 

\begin{document}

\title{A user's guide to co/cartesian fibrations}
\date{\today}
\author{\textsc{Aaron Mazel-Gee}}

\begin{abstract}
We formulate a model-independent theory of co/cartesian morphisms and co/cartesian fibrations: that is, one which resides entirely \textit{within the $\infty$-category of $\infty$-categories}.  We prove this is suitably compatible with the corresponding quasicategorical (and in particular, model-dependent) notions studied by Joyal and Lurie.
\end{abstract}

\maketitle

\vspace{-25pt}

\setcounter{tocdepth}{1}
\tableofcontents

\setcounter{section}{-1}

\vspace{-25pt}

\section{Introduction}

\subsection*{Outline and goals}

This note concerns the notions of \textit{cocartesian fibrations} and of \textit{cartesian fibrations}, as well as their connection to the \textit{Grothendieck construction}, in the $\infty$-categorical context.

\begin{itemize}

\item The ur-example of a \bit{cartesian fibration} is the forgetful functor $\VBdl \ra \Mfld$ from the category of vector bundles -- that is, of pairs $(M,V)$ of a manifold $M$ and a vector bundle $V \da M$ -- to the category of manifolds; the fact that this is a cartesian fibration encodes the observation that vector bundles can be \textit{pulled back} 
in a category-theoretically meaningful way.  (See \cref{ex vector bundles and tangent bundle}.)

\item Meanwhile, the ur-example of the \bit{Grothendieck construction} accounts for the equivalence between the two competing definitions of a \textit{stack} in algebraic geometry, either as a functor to the category of groupoids or as a ``category fibered in groupoids''.

\end{itemize}

The purpose of this note is twofold.
\begin{enumerate}

\item

On the one hand, we offer a relaxed and informal discussion of co/cartesian fibrations and their motivation coming from the Grothendieck construction, which assumes no prerequisites beyond a vague sense of the meaning of the signifier ``$\infty$-category''.

\item

On the other hand, we carefully prove that our relaxed and informal discussion was in fact actually \bit{completely rigorous all along}.  More precisely, we show that our definitions, which are formulated \textit{within the $\infty$-category of $\infty$-categories}, are suitably compatible with the corresponding notions in quasicategories (see \cref{invce of cart morphism} and \cref{invce of cart fibn}).\footnote{As the quasicategorical definitions are not manifestly model-independent, the proofs of these results are nontrivial.  Indeed, they are fairly involved, and rely on rather subtle model-categorical manipulations.}

\end{enumerate}
The bulk of this note consists in \cref{section ramble on co-cart and gr} (where we offer our relaxed and informal discussion) and in \cref{section of proofs} (where we present our proofs).  Separating these, in \cref{section of definitions} we distill the discussion of \cref{section ramble on co-cart and gr} into precise model-independent definitions of {co/cartesian morphisms} and of {co/cartesian fibrations}.

\subsection*{Wherefore model-independence?}

As for the technical content of this note, the skeptical reader is completely justified in asking: \textit{Why does this matter?}  We offer three related and complementary responses.

\begin{enumeratesub}

\item\label{allows to work invariantly}

The first and most obvious justification for model-independence is that it allows one to \bit{work model-independently}.  Being a vastly general and broadly applicable home for \textit{derived mathematics}, the theory of $\infty$-categories has recently found widespread and exciting use in a plethora of different areas -- for instance in the geometric Langlands program and in mathematical physics, to name but two (closely related) examples.  In such applications, ``$\infty$-categories'' are generally manipulated in a purely formal fashion, a rule of thumb being that anything the typical user of $\infty$-categories would like to study should be accessible without reference to model-dependent notions such as quasicategories, various sorts of fibrations between them, their individual simplices, etc.

For the most part, the theory of quasicategories -- being itself soundly founded in the theory of model categories -- allows for direct and straightforward manipulation of the corresponding model-independent notions: by and large, the definitions visibly descend to the \textit{underlying $\infty$-category} of the model category $s\Set_\Joyal$. 
However, the theory of co/cartesian fibrations is a glaring exception.  It is therefore \textit{not} a priori meaningful to work with these notions in a model-independent fashion.  In order to allow the myriad users of $\infty$-category theory throughout mathematics to employ the theory of co/cartesian fibrations -- and in particular, to allow them to appeal to the various results which have been proved about their incarnations in quasicategories --, it seems like nothing more than \textit{good homotopical manners} to prove that these quasicategorical definitions do indeed descend to the underlying $\infty$-category of $s\Set_\Joyal$ as well.

In particular, our results provide the crucial input to a model-independent reading of 
\cite{LurieHA}, which work is premised heavily on the notion of co/cartesian morphisms in quasicategories. 

\item

The next most obvious justification for model-independence is that it provides \bit{conceptual clarity}: a model-independent definition is \textit{by definition} unfettered by point-set or model-dependent notions which obfuscate its true meaning and significance.

For a rather grotesque example, recall that one can define the ``$n\th$ homotopy group'' of a based simplicial set $(Y,y) \in s\Set_*$ as certain subquotient
\[ \pi_n(Y,y) = \left. \left\{ \sigma \in (\bbR Y)_n : \delta^n_i(\sigma) = \sigma^{\circ (n-1)}(y) \textup{ for all } 0 \leq i \leq n \right\} \right/ \sim \]
of the set of $n$-simplices of a fibrant replacement $Y \we \bbR Y \fibn \Delta^0$ (relative to the standard Kan--Quillen model category structure $s\Set_\KQ$, and with respect to the induced basepoint $\Delta^0 \xra{y} Y \ra \bbR Y$).  This definition completely obscures a number of important features of homotopy groups:
\begin{itemizesmall}
\item that they are independent of the choice of fibrant replacement;
\item that they actually form a \textit{group} at all (for $n \geq 1$, let alone an \textit{abelian} group for $n \geq 2$);
\item that a path connecting two basepoints (which itself may only be representable by a \textit{zigzag} of edges in $Y$ itself) induces a \textit{conjugation isomorphism} (for $n \geq 1$);
\item that they are \textit{corepresentable} in the homotopy category $\ho(\S_*)$ of pointed spaces.
\end{itemizesmall}

Though the quasicategorical definitions of co/cartesian morphisms and co/cartesian fibrations are not nearly so abstruse, they are nevertheless model-dependent, and hence the assertion that they have \textit{homotopical meaning} requires further proof.

\item

Lastly, proving that quasicategorical definitions are model-independent also allows mathematicians employing \textit{specific} but \bit{alternate models for $\infty$-categories} to employ these notions and their attending results while continuing to work within their native context.  As is particularly well-known to homotopy theorists, different model categories presenting the same $\infty$-category (e.g.\! and i.e.\! that of spectra) can be advantageous for different purposes.  For a few examples of where such alternate models for $\infty$-categories arise:
\begin{itemizesmall}
\item $\Top$-enriched categories appear in geometric topology;
\item $\A_\infty$-categories appear in mirror symmetry;
\item dg-categories appear in derived algebraic geometry and homological algebra;
\item Segal spaces (and indeed, Segal $\Theta_n$-spaces) appear in bordism theory;
\item $s\Set$-enriched categories appear in homotopy theory (both as hammock localizations of relative categories (e.g.\! model categories) and as (the full subcategories of bifibrant objects of) simplicial model categories).
\end{itemizesmall}

In some sense, this is mainly a preservation-of-sanity issue.  All of the various the model categories which present ``the $\infty$-category of $\infty$-categories'' are connected by an intricate web of Quillen equivalences (see \cite{BSP}).  It is therefore already possible to wrangle a given question asked in any of these model categories into a corresponding question asked in the model category $s\Set_\Joyal$, so long as one is willing to take the appropriate co/fibrant replacements at every step to ensure that one is computing the \textit{derived} values of these various Quillen equivalences.  However, this procedure is a hassle at best, and at worst can effectively destroy all hope of answering the given question (since taking co/fibrant replacements generally drastically alters the object's point-set features).

This motivation can therefore be seen as something of a ``down-then-up'' maneuver: in proving that these notions actually descend to the underlying $\infty$-category which is common to \textit{all} of these distinct model categories of $\infty$-categories, we prove that all of the results which have been proved in quasicategories can be brought to bear while working in \textit{any} of these models -- no wrangling required.

\end{enumeratesub}

\subsection*{Conventions}

We refer to Lurie's twin tomes \cite{LurieHTT} and \cite{LurieHA} for general references on quasicategories.  For further motivation regarding the Grothendieck construction (including a variety of applications), we refer the interested reader to \cite{MIC-gr}.  Lastly, we refer to \cite[\sec A]{MIC-sspaces} for details and specifics on working ``within the $\infty$-category of $\infty$-categories'' (as well as for further clarification regarding our conventions).

\subsection*{Acknowledgments}

It is a pleasure to acknowledge our debt of influence to David Ayala, 
both for the generalities as well as for various specific observations that appear throughout this note.  We are also deeply grateful to Piotr Pstragowski for his suggested outline of the proof of \cref{invce of cart morphism}, and to Zhen Lin Low for countless conversations regarding the ideas that appear here.  Finally, we would like to thank Montana State University for its hospitality as well as UC Berkeley's Geometry and Topology RTG grant (which is part of NSF grant DMS-0838703) for its generous funding of our stay there (during which this note was written).

\section{A leisurely soliloquy on co/cartesian fibrations and the Grothendieck construction}\label{section ramble on co-cart and gr}

A fundamental principle in mathematics is that objects do not exist only in isolation, but tend to occur in \textit{families}.  Perhaps the most basic example is that a covering space
\[ \begin{tikzcd}
E \arrow{d}[swap]{f} \\ B
\end{tikzcd} \]
can be thought of as a family of sets, parametrized by the base space $B$: a point $b \in B$ corresponds to the fiber $f^{-1}(b) \subset E$.  This allows for a natural shift in perspective: our covering, a map whose \textit{target} is the space $B$, is simultaneously classified by a map whose \textit{source} is the space $B$, namely a map
\[ \begin{tikzcd}
B \arrow{r} & \Set^\simeq
\end{tikzcd} \]
to the maximal subgroupoid $\Set^\simeq \subset \Set$ of the category of sets.  Indeed, this construction furnishes an isomorphism
\[ \begin{tikzcd}
\C\textup{ov}(B) \arrow{r}{\cong} & {[ B , \Set^\simeq ]}
\end{tikzcd} \]
from the set of covering spaces of $B$ to the set of homotopy classes of maps $B \ra \Set^\simeq$.


But this, in turn, allows for another a shift in perspective.  Returning to our covering space, note that a path $b_1 \ra b_2$ between two points of $B$ provides an isomorphism $f^{-1}(b_1) \xra{\cong} f^{-1}(b_2)$ between their respective fibers.  In other words, it is only because \bit{all morphisms in a space are invertible} that our classifying map $B \ra \Set$ factors through the maximal subgroupoid $\Set^\simeq \subset \Set$.

\begin{qn}\label{qn if morphisms in space are not all invertible}
If $\B$ is a ``space whose morphisms are not all invertible'' -- that is, if $\B$ is an \textit{$\infty$-category} --, then what, exactly, is classified by a map $\B \ra \Set$?
\end{qn}

The answer to this question -- and to its successive generalizations along the inclusions
\[ \Set \subset \S \subset \Cati \]
of the category of sets into the $\infty$-categories $\S$ of spaces and $\Cati$ of $\infty$-categories -- is provided by the \bit{Grothendieck construction}, which furnishes an equivalence of $\infty$-categories
\[ \begin{tikzcd}
\Fun(\B,\Cati) \arrow{r}{\Gr}[swap]{\sim} & \coCartFib(\B)
\end{tikzcd} \]
from the $\infty$-category of functors $\B \ra \Cati$ to the $\infty$-category of \bit{cocartesian fibrations} over $\B$, a subcategory
\[ \coCartFib(\B) \subset (\Cati)_{/\B} \]
of the $\infty$-category of $\infty$-categories lying over $\B$.

Given a functor
\[ \begin{tikzcd}
\B \arrow{r}{F} & \Cati ,
\end{tikzcd} \]
let us describe the salient features of the resulting cocartesian fibration
\[ \begin{tikzcd}
\E \arrow{d}[swap]{f} \\ \B
\end{tikzcd} \]
which it classifies.
\begin{enumerate}

\item

Over an object $b \in \B$, the fiber $f^{-1}(b) \subset \E$ is canonically equivalent to $F(b) \in \Cati$.

\item\label{cocartesian lifts exist}

Given a morphism $b_1 \xra{\varphi} b_2$ in $\B$ and an object $e \in f^{-1}(b_1)$ of the fiber over its source, there is a canonical morphism
\[ e \ra \varphi_* (e) \]
in $\E$ which projects to $\varphi$, called an \textit{$f$-cocartesian lift} (or simply a \textit{cocartesian lift}) of $\varphi$ relative to $e$, such that the canonical equivalence $f^{-1}(b_2) \simeq F(b_2)$ identifies the object $\varphi_*(e) \in f^{-1}(b_2)$ with the object $(F\varphi)(e) \in F(b_2)$.  This is illustrated in \cref{diagram of cocartesian lift}.
\begin{figure}[h]
\[ \begin{tikzcd}[row sep=0.5cm, column sep=0.5cm]
e \arrow[dashed]{r} & \varphi_*(e) \\
& & \E \arrow{dddd}[swap]{f} \\
\ \\
\ \\
\ \\
& & \B \arrow{rrrr}[swap]{F} & & & & \Cati \\
b_1 \arrow{r}{\varphi} & b_2 & \ & \ & \ & \ & \ & F(b_1) \arrow{r}{F\varphi} & F(b_2) \\
& & & & & & & e \arrow[mapsto]{r} & (F\varphi)(e)
\end{tikzcd} \]
\caption{An illustration of a cocartesian morphism.}
\label{diagram of cocartesian lift}
\end{figure}

\item\label{factorization system in cocart fibn}

An arbitrary morphism $e_1 \ra e_2$ in $\E$ admits a \textit{unique factorization} as a cocartesian morphism followed by a morphism lying in the fiber $f^{-1}(b_2)$ -- which we will therefore refer to as a \textit{fiber morphism} --, as illustrated in \cref{diagram of a cocart-fib factorizn}.
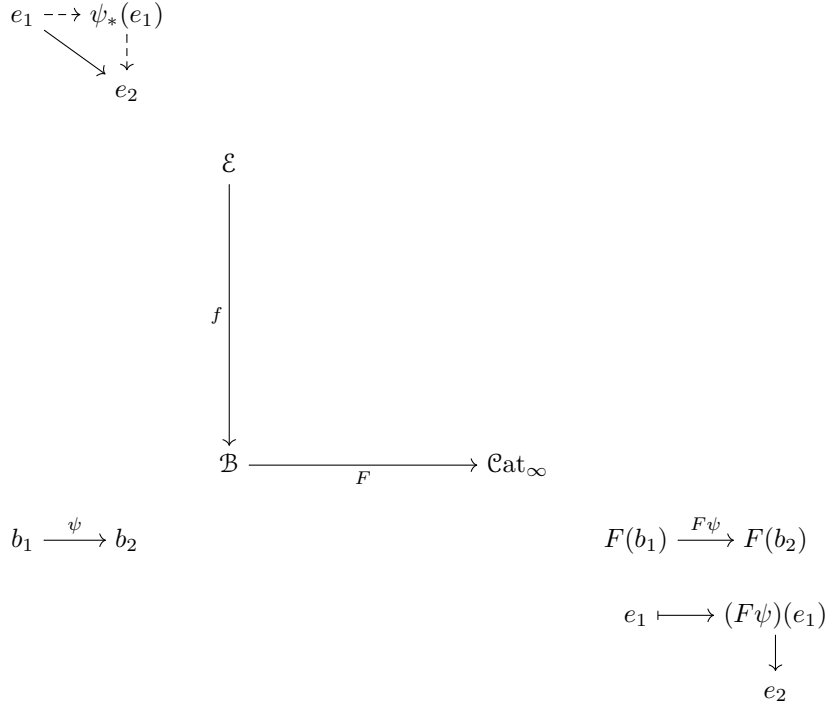
\begin{figure}[h]
\[ \begin{tikzcd}[row sep=0.5cm, column sep=0.5cm]
e_1 \arrow{rd} \arrow[dashed]{r} & \psi_*(e_1) \arrow[dashed]{d} \\
& e_2 \\
& & \E \arrow{dddd}[swap]{f} \\
\ \\
\ \\
\ \\
& & \B \arrow{rrrr}[swap]{F} & & & & \Cati \\
b_1 \arrow{r}{\psi} & b_2 & \ & \ & \ & \ & \ & F(b_1) \arrow{r}{F\psi} & F(b_2) \\
& & & & & & & e_1 \arrow[mapsto]{r} & (F\psi)(e_1) \arrow{d} \\
& & & & & & & & e_2
\end{tikzcd} \]

\caption{An illustration of the factorization system in a cocartesian fibration.}
\label{diagram of a cocart-fib factorizn}
\end{figure}
Under the equivalence $f^{-1}(b_2) \simeq F(b_2)$, the morphism $\psi_*(e_1) \ra e_2$ in $f^{-1}(b_2)$ corresponds to a morphism $(F\psi)(e_1) \ra e_2$ in $F(b_2)$.  Thus, we have canonical equivalences
\[ \hom_\E(e_1,e_2) \simeq \hom_{f^{-1}(b_2)}(\psi_*(e_1),e_2) \simeq \hom_{F(b_2)}((F\psi)(e_1),e_2) \]
of hom-spaces.

\end{enumerate}

This informal description already makes visible an exciting feature of the Grothendieck construction, namely that it \textit{reduces category level}.  For instance, in \cref{diagram of cocartesian lift}, we see that the Grothendieck construction translates
\[ e \mapsto (F\varphi)(e) , \]
an assertion about a functor \textit{between $\infty$-categories}, into
\[ e \ra \varphi_*(e) , \]
a morphism \textit{within a single $\infty$-category}.

\begin{ex}\label{example of lax symm mon fctrs and algebra in particular}
Let us illustrate just a hint of the bookkeeping power which results from this reduction of category level.  First of all, the datum of a \textit{monoidal} category $(\C,\otimes)$ can be encoded as a certain functor
\[ \begin{tikzcd}[column sep=1.5cm]
\bD^{op} \arrow{r}{\Bar(\C)_\bullet} & \Cat ,
\end{tikzcd} \]
namely its \textit{bar construction} (as a monoid object in the symmetric monoidal category $(\Cat,\times)$): this is given on objects by $\Bar(\C)_n = \C^{\times n}$, while its structure maps encode the monoidal structure on $\C$ and the unit map $\pt_\Cat \simeq \{ \unit_\C \} \hookra \C$.  Similarly, we can encode the datum of a \textit{symmetric monoidal} category $(\C,\otimes)$ as a functor
\[ \Fin_* \ra \Cat \]
from the category of finite pointed sets: this takes an object $T_+ = T \sqcup \{ * \}$ to the category $\C^{\times T}$, and it takes a morphism $T_+ \xra{\alpha} U_+$ to the functor $\C^{\times T} \ra \C^{\times U}$ described by the formula
\[ ( c_t )_{t \in T} \mapsto \left( \bigotimes_{t \in \alpha^{-1}(u)} c_t \right)_{u \in U} \]
where, by convention, a monoidal product indexed over an empty set is defined to be the unit object $\unit_\C \in \C$.  (We note in passing that if $\alpha(t) = * \in U_+$ for some $t \in T$, then $c_t$ does not appear in any of the resulting monoidal products indexed by the elements $u \in U$: it is simply ``thrown away''.)  It follows that we can equivalently consider a symmetric monoidal category $(\C,\otimes)$ as a cocartesian fibration
\[ \begin{tikzcd}
\C^\otimes \arrow{d}[swap]{p} \\ \Fin_*
\end{tikzcd} \]
over the category of finite pointed sets.  For instance, writing $\brax{n} = \{ 1, \ldots, n\}_+$, the unique map $\brax{2} \xra{\mu} \brax{1}$ in $\Fin_*$ satisfying $\mu^{-1}(*) = \{ * \}$ has cocartesian lifts of the form
\[ (c_1,c_2) \ra (c_1 \otimes c_2) , \]
which morphisms therefore encode the monoidal product $\C \times \C \xra{- \otimes -} \C$.

Now, in this language, a \textit{symmetric monoidal functor}
\[ (\C,\otimes) \xra{F} (\D,\boxtimes) \]
is equivalent data to that of a morphism of cocartesian fibrations
\[ \begin{tikzcd}
\C^\otimes \arrow{rr}{F} \arrow{rd}[swap]{p} & & \D^\boxtimes \arrow{ld}{q} \\
& \Fin_*
\end{tikzcd} \]
over $\Fin_*$: over an object $T_+ \in \Fin_*$ the induced map on fibers is given by $\C^{\times n} \xra{F^{\times n}} \D^{\times n}$, and the fact that the functor respects the monoidal products is encoded by the fact that it preserves cocartesian morphisms.  For instance, the $p$-cocartesian morphism
\[ (c_1,c_2) \ra (c_1 \otimes c_2) \]
of $\C^\otimes$ lying over $\brax{2} \xra{\mu} \brax{1}$ is taken to a morphism
\[ (F(c_1),F(c_2)) \ra F(c_1 \otimes c_2) \]
of $\D^\boxtimes$ also lying over $\brax{2} \xra{\mu} \brax{1}$, and the assertion that this is $q$-cocartesian guarantees a unique isomorphism
\[ F(c_1) \boxtimes F(c_2) \cong F(c_1 \otimes c_2) \]
in $\D \cong \D^{\times 1}$.

However, more is true: even if $F$ is now only a \textit{lax symmetric monoidal functor}, it still defines a morphism \textit{among} cocartesian fibrations (in constrast to ``a morphism \textit{of} cocartesian fibrations''), but in general it will \textit{not} preserve the cocartesian morphisms.  For instance, the $p$-cocartesian morphism $(c_1,c_2) \ra (c_1 \otimes c_2)$ in $\C^\otimes$ will be sent to an \textit{arbitrary} morphism $(F(c_1),F(c_2)) \ra F(c_1 \otimes c_2)$, which then admits a unique cocartesian/fiber factorization
\[ \begin{tikzcd}[row sep=0.5cm, column sep=0.5cm]
(F(c_1),F(c_2)) \arrow[dashed]{r} \arrow{rd} & (F(c_1) \boxtimes F(c_2)) \arrow[dashed]{d} \\
& F(c_1 \otimes c_2)
\end{tikzcd} \]
in $\D^\boxtimes$, in which the fiber morphism is the ``structure map'' witnessing the laxness of $F$ (at the pair of objects $c_1,c_2 \in \C$).

As a special case, note that the identity map of $\Fin_*$ is a cocartesian fibration, which corresponds to the canonical (and unique) symmetric monoidal structure on the terminal category $\pt_\Cat \in \Cat$.  Then, a \textit{commutative algebra object} in the symmetric monoidal category $(\C,\otimes)$ is nothing but a lax symmetric monoidal functor
\[ \pt_\Cat \xra{A} (\C,\otimes) . \]
Moreover, the functoriality of commutative algebra objects for a lax symmetric monoidal functor
\[ (\C,\otimes) \xra{F} (\D,\boxtimes) \]
is encoded simply by the composition
\[ \begin{tikzcd}
\Fin_* \arrow{r}{A} \arrow{rd}[swap]{\id_{\Fin_*}} & \C^\otimes \arrow{r}{F} \arrow{d}[swap]{p} & \D^\boxtimes \arrow{ld}{q} \\
& \Fin_*
\end{tikzcd} \]
of morphisms among cocartesian fibrations.  (Of course, we can equivalently write the map $A$ as a \textit{section} of the cocartesian fibration $\C^\otimes \ra \Fin_*$; then, the functoriality of commutative algebras for lax symmetric monoidal functors is encoded by the functoriality of sections.)
\end{ex}

\begin{rem}\label{need to define inert}
In \cref{example of lax symm mon fctrs and algebra in particular}, we were careful not to say that a lax symmetric monoidal functor (or, in particular, an algebra object) was exactly \textit{characterized} as a morphism
\[ \begin{tikzcd}
\C^\otimes \arrow{rr}{F} \arrow{rd}[swap]{p} & & \D^\boxtimes \arrow{ld}{q} \\
& \Fin_*
\end{tikzcd} \]
among cocartesian fibrations.  This was only because we had not yet introduced a certain bit of terminology.  First of all, let us say that a morphism $T_+ \xra{\alpha} U_+$ in $\Fin_*$ is \textit{inert} if for all $u \in U$, the preimage $\alpha^{-1}(u)$ has exactly one element.  Inasmuch as the basepoints of objects in $\Fin_*$ might be thought of as ``trash receptacles'', such an inert morphism should therefore be thought of as parametrizing the operation of ``throwing away'' some specified subset of the $T$-indexed list of objects.  For instance, the unique map $\brax{2} \xra{\rho} \brax{1}$ in $\Fin_*$ satisfying $\rho^{-1}(1) = \{ 2 \}$ has $p$-cocartesian lifts of the form
\[ (c_1,c_2) \ra c_2 . \]
Then, a morphism in $\C^\otimes$ is called \textit{$p$-inert} (or simply \textit{inert}) if it is $p$-cocartesian and lies over an inert morphism in $\Fin_*$.  Now, we can \textit{define} a lax symmetric monoidal functor
\[ (\C,\otimes) \xra{F} (\D,\boxtimes) \]
to be a commutative triangle as above which preserves inert morphisms (i.e.\! which takes $p$-inert morphisms to $q$-inert morphisms).
\end{rem}

\begin{rem}
The observations of \cref{example of lax symm mon fctrs and algebra in particular} and \cref{need to define inert} form the foundations of the extremely versatile theory of \textit{$\infty$-operads} introduced and studied in \cite[Chapter 2]{LurieHA}.
\end{rem}

In \cref{qn if morphisms in space are not all invertible}, we made an implicit choice when generalizing morphisms
\[ \begin{tikzcd}
B \arrow{r} & \Set^\simeq
\end{tikzcd} \]
of $\infty$-groupoids to morphisms
\[ \begin{tikzcd}
\B \arrow{r} & \Set
\end{tikzcd} \]
of $\infty$-categories: we chose to study \textit{covariant} functors.  There is also a \textit{contravariant} Grothendieck construction
\[ \begin{tikzcd}
\Fun(\B^{op},\Cati) \arrow{r}{\Grop}[swap]{\sim} & \CartFib(\B)
\end{tikzcd} \]
from the $\infty$-category of functors $\B^{op} \ra \Cati$ to the $\infty$-category of \bit{cartesian fibrations} over $\B$, which is likewise a subcategory
\[ \CartFib(\B) \subset (\Cati)_{/\B} \]
of the $\infty$-category of $\infty$-categories lying over $\B$.  In parallel with the salient features of a cocartesian fibration described above, let us briefly describe those of the cartesian fibration
\[ \begin{tikzcd}
E \arrow{d}[swap]{f} \\ \B
\end{tikzcd} \]
classified by a functor
\[ \begin{tikzcd}
\B^{op} \arrow{r}{F} & \Cati .
\end{tikzcd} \]

\begin{enumerate}

\item Over an object $b \in \B$, the fiber $f^{-1}(b) \subset \E$ is canonically equivalent to $F(b) \in \Cati$.

\item Given a morphism $b_1 \xra{\varphi} b_2$ in $\B$ and an object $e \in f^{-1}(b_2)$ of the fiber over its \textit{target}, there is a canonical morphism
\[ \varphi^*(e) \ra e \]
in $\E$ which projects to $\varphi$, called a \textit{cartesian lift} of $\varphi$ (relative to $e$).

\item An arbitrary morphism $e_1 \ra e_2$ in $\E$ projecting to a map $b_1 \xra{\psi} b_2$ in $\B$ now admits a unique factorization
\[ \begin{tikzcd}[row sep=0.5cm, column sep=0.5cm]
e_1 \arrow{rd} \arrow[dashed]{d} \\
\psi^*(e_2) \arrow[dashed]{r} & e_2
\end{tikzcd} \]
as a fiber morphism followed by a cartesian morphism.

\end{enumerate}

\begin{ex}\label{ex vector bundles and tangent bundle}
Consider the category $\VBdl$ of vector bundles: its objects are the pairs $(M , V)$ of a manifold $M$ and a vector bundle $V \da M$, and its morphisms are commutative squares
\[ \begin{tikzcd}
V \arrow{r} \arrow{d} & W \arrow{d} \\
M \arrow{r} & N
\end{tikzcd} \]
(of a morphism of manifolds and a compatible morphism of (total spaces of) vector bundles).  Then, the forgetful functor
\[ \begin{tikzcd}
\VBdl \arrow{d}[swap]{\forget_\VBdl} \\ \Mfld
\end{tikzcd} \]
to the category of manifolds is a cartesian fibration.  Given a morphism $M \xra{\varphi} N$ in $\Mfld$ and a vector bundle $W \da N$, a cartesian lift is provided by the \textit{pullback vector bundle}, which is defined by a pullback square
\[ \begin{tikzcd}
\varphi^*(W) \arrow{r} \arrow{d} & W \arrow{d} \\
M \arrow{r}[swap]{\varphi} & N
\end{tikzcd} \]
of underlying topological spaces.  The fiber/cartesian factorization system in this cartesian fibration translates into the assertion that for any vector bundle $V \da M$, the induced diagram
\[ \begin{tikzcd}[row sep=1.5cm]
\hom_{\VBdl(M)}(V \da M , \varphi^*(M) \da M) \arrow{r} \arrow{d} & \hom_\VBdl(V \da M , W \da N) \arrow{d} \\
\{ \varphi \} \arrow[hook]{r} & \hom_\Mfld(M,N)
\end{tikzcd} \]
is a pullback square in $\Set$, where $\VBdl(M)$ denotes the fiber $\forget_\VBdl^{-1}(M)$ of the forgetful functor over the object $M \in \Mfld$ (or equivalently, the value at $M$ of the functor
\[ \begin{tikzcd}[column sep=1.5cm]
\Mfld^{op} \arrow{r}{\VBdl} &  \Cat
\end{tikzcd} \]
 classifying our cartesian fibration).

There is a canonical section
\[ \begin{tikzcd}
\VBdl \arrow{d}[swap]{\forget_\VBdl} \\ \Mfld \arrow[dashed, bend right=50]{u}[swap, pos=0.45]{T}
\end{tikzcd} \]
of the forgetful functor, called the \textit{tangent bundle}: this takes a manifold $M \in \Mfld$ to its tangent bundle $TM \da M$.  Note that this is \textit{not} a cartesian section: it does not take morphisms in $\Mfld$ to cartesian morphisms of the cartesian fibration.  (Correspondingly, this does not arise from a natural transformation
\[ \begin{tikzcd}[column sep=1.5cm]
\Mfld^{op} \arrow[bend left=50]{r}{\const(\pt_\Cat)}[swap, transform canvas={yshift=-1.5em}]{\Downarrow} \arrow[bend right=50]{r}[swap]{\VBdl} & \Cat
\end{tikzcd} \]
between $\Cat$-valued functors on $\Mfld^{op}$.)  In other words, given an arbitrary morphism $M \xra{\varphi} N$ of manifolds, we obtain a canonical map $TM \ra \varphi^*(TN)$ in $\VBdl(M)$, but this map is not generally an isomorphism.  However, it is an isomorphism (and the morphism $TM \ra \varphi^*(TN)$ is a cartesian lift of $\varphi$) whenever the morphism $\varphi$ is an \textit{open embedding}.
\end{ex}

The following example illustrates the essential consequence on hom-sets (or hom-spaces) of a functor being a co/cartesian fibration, which was alluded to in \cref{ex vector bundles and tangent bundle} (but in which case the notation would have been a bit unwieldy if we had tried to elaborate fully on this consequence there).

\begin{ex}\label{ex forget Top to Set}
The forgetful functor
\[ \begin{tikzcd}
\Top \arrow{d}[swap]{\forget_\Top} \\ \Set
\end{tikzcd} \]
is a cartesian fibration.  Given a morphism $U \ra Y$ in $\Set$ and a topological space $\Y \in \Top$ equipped with an isomorphism $\forget_\Top(\Y) \cong Y$ in $\Set$, a $\forget_\Top$-cartesian lift is provided by endowing the set $U \in \Set$ with the \textit{induced topology}: this yields a topological space $\U \in \Top$ equipped with a map $\U \ra \Y$ in $\Top$ and an isomorphism $\forget_\Top(\U) \cong U$ in $\Set$, which has the universal property that for any $\Z \in \Top$ with underlying set $Z = \forget_\Top(\Z) \in \Set$, the resulting diagram
\[ \begin{tikzcd}[row sep=1.5cm]
\hom_\Top(\Z,\U) \arrow{r} \arrow{d} & \hom_\Top(\Z,\Y) \arrow{d} \\
\hom_\Set( Z , U ) \arrow{r} & \hom_\Set(Z,Y)
\end{tikzcd} \]
is a pullback square in $\Set$.  (If the morphism $U \ra Y$ is actually the inclusion of a subset, this specializes to define the \textit{subspace topology} on the set $U$.)
\end{ex}

\section{Definitions}\label{section of definitions}

In this brief section, we give precise definitions of the concepts we have introduced in \cref{section ramble on co-cart and gr}.

\begin{defn}
Let $\E \xra{f} \B$ be a functor of $\infty$-categories.
\begin{enumerate}
\item
We say that a morphism $e_1 \xra{\varphi} e_2$ in $\E$ is \bit{$f$-cocartesian} (or simply \bit{cocartesian}, if the functor $f$ is clear from the context) if it induces a pullback diagram
\[ \begin{tikzcd}[column sep=1.5cm, row sep=1.5cm]
\E_{e_2/} \arrow{r}{-\circ \varphi} \arrow{d}[swap]{f} & \E_{e_1/} \arrow{d}{f} \\
\B_{f(e_2)/} \arrow{r}[swap]{-\circ f(\varphi)} & \B_{f(e_1)/}
\end{tikzcd} \]
in $\Cati$.  In this case, we call $\varphi$ an \bit{$f$-cocartesian lift} (or simply a \bit{cocartesian lift}) of $f(\varphi)$ relative to $e_1$.  We then say that the functor $f$ is a \bit{cocartesian fibration} if every object of
\[ \Fun([1],\B) \underset{s,\B,f}{\times} \E \]
admits an $f$-cocartesian lift.
\item
Dually, we say that a morphism $e_1 \xra{\varphi} e_2$ in $\E$ is \bit{$f$-cartesian} (or simply \bit{cartesian}) if it induces a pullback diagram
\[ \begin{tikzcd}[column sep=1.5cm, row sep=1.5cm]
\E_{/e_1} \arrow{r}{\varphi \circ -} \arrow{d}[swap]{f} & \E_{/e_2} \arrow{d}{f} \\
\B_{/f(e_1)} \arrow{r}[swap]{f(\varphi) \circ -} & \B_{/f(e_2)}
\end{tikzcd} \]
in $\Cati$.  In this case, we call $\varphi$ an \bit{$f$-cartesian lift} (or simply a \bit{cartesian lift}) of $f(\varphi)$ relative to $e_2$.  We then say that the functor $f$ is a \bit{cartesian fibration} if every object of
\[ \Fun([1],\B) \underset{t,\B,f}{\times} \E \]
admits an $f$-cartesian lift.
\end{enumerate}
\end{defn}

\begin{rem}
Let $\E \xra{f} \B$ be a functor of $\infty$-categories.  The crucial consequence of a morphism $e_1 \xra{\varphi} e_2$ in $\E$ being $f$-cocartesian is that for any object $e \in \E$, the induced diagram
\[ \begin{tikzcd}[row sep=1.5cm]
\hom_\E(e_2,e) \arrow{r} \arrow{d} & \hom_\E(e_1,e) \arrow{d} \\
\hom_\B(f(e_2),f(e)) \arrow{r} & \hom_\B(f(e_1),f(e))
\end{tikzcd} \]
is a pullback square in $\S$.  Dually, the crucial consequence of a morphism $e_1 \xra{\varphi} e_2$ in $\E$ being $f$-cartesian is that for any object $e \in \E$, the induced diagram
\[ \begin{tikzcd}[row sep=1.5cm]
\hom_\E(e,e_1) \arrow{r} \arrow{d} & \hom_\E(e,e_2) \arrow{d} \\
\hom_\B(f(e),f(e_1)) \arrow{r} & \hom_\B(f(e),f(e_2))
\end{tikzcd} \]
is a pullback square in $\S$.  (Recall \cref{ex forget Top to Set}, and see \cite[Proposition 2.4.1.10]{LurieHTT}.)
\end{rem}

\begin{rem}
A cocartesian fibration $\E \xra{f} \B$ whose fibers $f^{-1}(b)$ are all $\infty$-groupoids is called a \textit{left fibration}.  These assemble into the full subcategory $\LFib(\B) \subset \coCartFib(\B)$.  Correspondingly, the (covariant) Grothendieck construction restricts to an equivalence
\[ \begin{tikzcd}
\Fun(\B,\Cati) \arrow{r}{\Gr}[swap]{\sim} & \coCartFib(\B) \\
\Fun(\B,\S) \arrow[hook]{u} \arrow{r}{\sim}[swap]{\Gr} & \LFib(\B) \arrow[hook]{u}
\end{tikzcd} \]
from the $\infty$-category of functors $\B \ra \S$ valued in the $\infty$-category of spaces to the $\infty$-category $\LFib(\B)$ of left fibrations over $\B$.  As a result of the fact that all morphisms in a space are equivalences, given a left fibration $\E \xra{f} \B$, \textit{every} morphism in $\E$ is $f$-cocartesian.

Dually, a cartesian fibration whose fibers are all $\infty$-groupoids is called a \textit{right fibration}, the contravariant Grothendieck construction restricts to an equivalence
\[ \begin{tikzcd}
\Fun(\B^{op},\Cati) \arrow{r}{\Grop}[swap]{\sim} & \CartFib(\B) \\
\Fun(\B^{op},\S) \arrow[hook]{u} \arrow{r}{\sim}[swap]{\Grop} & \RFib(\B) \arrow[hook]{u}
\end{tikzcd} \]
from the $\infty$-category of functors $\B^{op} \ra \S$ to the full subcategory $\RFib(\B) \subset \CartFib(\B)$ of right fibrations over $\B$, and given a right fibration $\E \xra{f} \B$, every morphism in $\E$ is $f$-cartesian.
\end{rem}

\section{Proofs}\label{section of proofs}

In this final section, we prove that our definitions of co/cartesian morphisms and co/cartesian fibrations coincide with the corresponding ``point-set'' definitions in quasicategories; as these latter have been studied extensively, it will follow that all quasicategorical results regarding co/cartesian morphisms and co/cartesian fibrations can be applied either when working model-independently or when working in some other model category of $\infty$-categories.  In order to directly align with the definitions given in \cite[\sec 2.4]{LurieHTT} we will focus on the cartesian variants, but of course our results will immediately apply to the cocartesian variants as well (simply by taking opposites).

\begin{notn}
We will write
\[ \enrhom(-,-) = \enrhom_{s\Set}(-,-) : (s\Set)^{op} \times s\Set \ra s\Set \]
for the internal hom bifunctor in $s\Set$ (relative to its cartesian symmetric monoidal structure).
\end{notn}

For precision and disambiguation, we introduce the following terminology (which pays homage to Joyal and Lurie, the architects of a theory without which the present note could never exist).

\begin{defn}
Let $\ttE \xra{\ttp} \ttB$ be an inner fibration in $s\Set$.  Following \cite[Definition 2.4.1.1]{LurieHTT}, we will say that an edge $\Delta^1 \xra{\ttf} \ttE$ is \bit{JL-$\ttp$-cartesian} (or simply \bit{JL-cartesian}) if the induced map
\[ \ttE_{/\ttf} \ra \ttE_{/\ttf(1)} \underset{\ttB_{/\ttp(\ttf(1))}}{\times}  \ttB_{/\ttp(\ttf)}  \]
lies in $(\bW \cap \bF)_\Joyal \subset s\Set$.\footnote{The map $\ttE_{/\ttf} \ra \ttE_{/\ttf(1)}$ should be thought of simply as ``postcomposition with $\ttf$'': the restriction map $\ttE_{/\ttf} \ra \ttE_{/\ttf(0)}$ lies in $(\bW \cap \bF)_\Joyal$.  Similarly for the map $\ttB_{/\ttp(\ttf)} \ra \ttB_{/\ttp(\ttf(1))}$.}  In this case, we will refer to the edge $\ttf \in \ttE_1$ as a \bit{JL-$\ttp$-cartesian lift} (or simply a \bit{JL-cartesian lift}) of the edge $\ttp(\ttf) \in \ttB_1$ relative to the vertex $\ttf(1) \in \ttE_0$.  Following \cite[Definition 2.4.2.1]{LurieHTT}, we then say that the morphism $\ttf$ is a \bit{JL-cartesian fibration} if every vertex of
\[ \enrhom(\Delta^1,\ttB) \underset{\ev_1,\ttB,\ttf}{\times} \ttE \]
admits an $\ttf$-cartesian lift.
\end{defn}

We now show that our model-independent notion of cartesian morphism is suitably compatible with the quasicategorical notion of a JL-cartesian edge.

\begin{thm}\label{invce of cart morphism}
Let $\ttE \stackrel{\ttp}{\fibn} \ttB$ be a fibration between fibrant objects in $s\Set_\Joyal$ which presents a map $\E \xra{p} \B$ in $\Cati$.  Suppose that a map $\Delta^1 \xra{\ttf} \ttE$ in $s\Set_\Joyal$ presents a $p$-cartesian morphism $[1] \xra{f} \E$.  Then the edge $\ttf \in \ttE_1$ is JL-$\ttp$-cartesian.
\end{thm}

\begin{proof}
We must show that the map
\[ \ttE_{/\ttf} \ra \ttE_{/\ttf(1)} \underset{\ttB_{/\ttp(\ttf(1))}}{\times} \ttB_{/\ttp(\ttf)} \]
in $s\Set_\Joyal$ lies in $(\bW \cap \bF)_\Joyal$.

First of all, using the characterization $\bF_\Joyal = \rlp((\bW \cap \bC)_\Joyal)$, it is easy to see
\begin{itemize}
\item that the object $\ttE_{/\ttf(1)} \in s\Set_\Joyal$ is fibrant and
\item that the map $\ttB_{/\ttp(\ttf)} \ra \ttB_{/\ttp(\ttf(1))}$ is a fibration in $s\Set_\Joyal$.
\end{itemize}
Hence, it follows from the Reedy trick that this fiber product is in fact a homotopy pullback in $s\Set_\Joyal$.  Thus, this map in $s\Set_\Joyal$ presents the map
\[ \E_{/f} \ra \E_{/f(1)} \underset{\B_{/p(f(1))}}{\times} \B_{/p(f)} , \]
in $\Cati$, which can be canonically identified with the map
\[ \E_{/f(0)} \ra \E_{f(1)} \underset{\B_{/p(f(1))}}{\times} \B_{/p(f(0))} \]
in $\Cati$ and therefore lies in $\bW_\Joyal$ by assumption.

To see that it also lies in $\bF_\Joyal$, we argue as follows.  We claim that there is a Quillen adjunction
\[ \alpha : \Fun([1],s\Set_\Joyal)_\Reedy \adjarr (s\Set_\Joyal)_{\Delta^1/} : \beta , \]
where
\begin{itemize}
\item we equip $[1]$ with the Reedy category structure determined by the degree function $0 \mapsto 0$ and $1 \mapsto 1$,
\item we define
\[ \beta ( \Delta^1 \xra{\tty} \ttY ) = ( \ttY_{/\tty} \to \ttY_{/\tty(1)} ) , \]
and
\item we define
\[ \alpha( \ttZ \ra \ttW ) = \left( \Delta^1 \ra \left( \ttZ \star \Delta^1 \coprod_{\ttZ \star \Delta^{\{1\}}} \ttW \star \Delta^{\{1\}} \right) \right) . \]
\end{itemize}
It is not hard to see that indeed $\alpha \adj \beta$, so it suffices to show that $\alpha$ is a left Quillen functor.  For this, given a map
\[ \begin{tikzcd}
\ttZ_1 \arrow{r}{\ttg_1} \arrow{d} & \ttW_1 \arrow{d} \\
\ttZ_2 \arrow{r}[swap]{\ttg_2} & \ttW_2
\end{tikzcd} \]
in $\Fun([1],s\Set_\Joyal)_\Reedy$ (reading the square vertically), observe that for this to be a (resp.\! acyclic) cofibration is precisely to require that the two relative latching maps $\ttZ_1 \ra \ttZ_2$ and $\ttW_1 \coprod_{\ttZ_1} \ttZ_2 \ra \ttW_2$ are (resp.\! acyclic) cofibrations.  For simplicity, let us write the composite
\[ \Fun([1],s\Set) \xra{\alpha} s\Set_{\Delta^1/} \ra s\Set \]
of our left adjoint with the evident forgetful functor simply as $\Fun([1],s\Set) \xra{\alpha'} s\Set$.  Now, assuming our map $\ttg_1 \ra \ttg_2$ is a cofibration in $\Fun([1],s\Set_\Joyal)_\Reedy$, then its image $\alpha'(\ttg_1) \ra \alpha'(\ttg_2)$ fits into the diagram
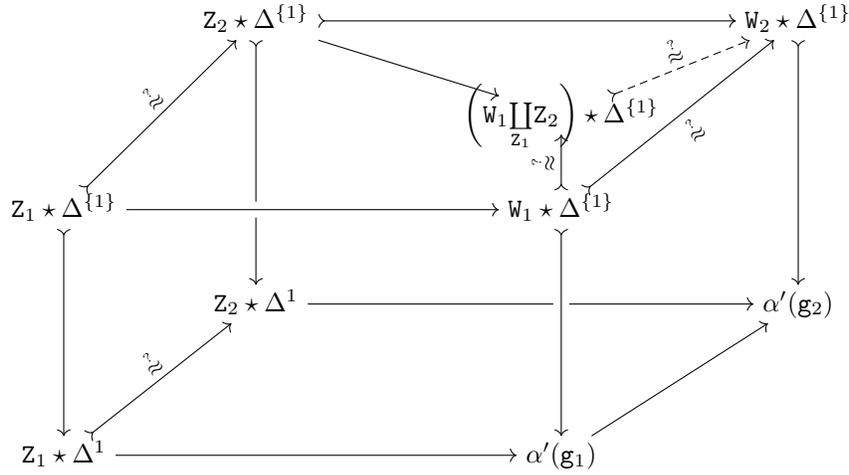
\begin{figure}[h]
\[ \begin{tikzcd}
 & \ttZ_2 \star \Delta^{\{1\}} \arrow[tail]{rrr} \arrow[tail]{ddd} \arrow{rrd} & & & \ttW_2 \star \Delta^{\{1\}} \arrow[tail]{ddd} \\
 & & &    \left( \ttW_1 \underset{\ttZ_1}{\coprod} \ttZ_2 \right) \star \Delta^{\{1\}}    \arrow[dashed, tail]{ru}[sloped, pos=0.6]{\stackrel{?}{\approx}} \\
 \ttZ_1 \star \Delta^{\{1\}} \arrow[tail]{ruu}[sloped, pos=0.6]{\stackrel{?}{\approx}} \arrow[crossing over]{rrr} \arrow[tail]{ddd} & & & \ttW_1 \star \Delta^{\{1\}} \arrow[tail]{u}[sloped, anchor=south]{\stackrel{?}{\approx}} \arrow[tail]{ruu}[sloped, swap, pos=0.45]{\stackrel{?}{\approx}} \\
 & \ttZ_2 \star \Delta^1 \arrow{rrr} & & & \alpha'(\ttg_2) \\ \\
\ttZ_1 \star \Delta^1 \arrow[tail]{ruu}[sloped, pos=0.6]{\stackrel{?}{\approx}} \arrow{rrr} & & & \alpha'(\ttg_1) \arrow{ruu} \arrow[leftarrowtail, crossing over]{uuu}
\end{tikzcd} \]
\caption{The diagram in $s\Set_\Joyal$ used in the proof of \cref{invce of cart morphism}.}
\label{diagram in proof of invce of cart morphism}
\end{figure}
in $s\Set_\Joyal$ of \cref{diagram in proof of invce of cart morphism}, in which
\begin{itemize}
\item the front and back faces are pushouts by definition;
\item the quadrilateral contained in the top face is a pushout since in the composite
\[ s\Set \xra{ - \star \Delta^{\{1\}}} s\Set_{\Delta^{\{1\}}/} \ra s\Set \]
where the second functor is forgetful,
\begin{itemize}
\item the first functor commutes with colimits by \cite[Remark 1.2.8.2]{LurieHTT} and
\item the second functor commutes with pushouts since the walking span $\Nerve^{-1}(\Lambda^2_0) \in \Cat$ has an initial object
\end{itemize}
(although really we have only rewritten this pushout to improve readability), and the dotted arrow is then the induced map;
\item the left face is a pushout by inspection;
\item all maps labeled as cofibrations are such
\begin{itemize}
\item by inspection,
\item by the assumption that $\ttg_1 \ra \ttg_2$ is a cofibration in $\Fun([1],s\Set_\Joyal)_\Reedy$,
\item because $\bC_\Joyal \subset s\Set$ is closed under pushouts, or
\item because $\bC_\Joyal \subset s\Set$ is closed under composition;
\end{itemize}
and
\item the maps labeled with the symbol $\stackrel{?}{\approx}$ are weak equivalences in $s\Set_\Joyal$ if $\ttg_1 \ra \ttg_2$ is additionally a weak equivalence in $\Fun([1],s\Set_\Joyal)_\Reedy$
\begin{itemize}
\item by the assumption that $\ttg_1 \ra \ttg_2$ is an acyclic cofibration in $\Fun([1],s\Set_\Joyal)_\Reedy$,
\item because $(\bW \cap \bC)_\Joyal \subset s\Set$ is closed under pushouts, or
\item because $(\bW \cap \bC)_\Joyal \subset s\Set$ is closed under composition.
\end{itemize}
\end{itemize}
Now, because the left and back faces are both pushouts, then the composite rectangle which they form is also a pushout.  But this is the same as the composite rectangle formed by the front and right faces.  As the front face is a pushout, it follows that the right face is also a pushout as well.  Thus, the functor
\[ \Fun([1],s\Set_\Joyal)_\Reedy \xra{\alpha'} s\Set_\Joyal \]
preserves both cofibrations and acyclic cofibrations, since these are each closed under pushout in $s\Set_\Joyal$.  But the cofibrations and acyclic cofibrations in $(s\Set_\Joyal)_{\Delta^1/}$ are created by the forgetful functor $s\Set_{\Delta^1/} \ra s\Set_\Joyal$, and so the functor
\[ \Fun([1],s\Set_\Joyal)_\Reedy \xra{\alpha} (s\Set_\Joyal)_{\Delta^1/} \]
is indeed a left Quillen functor, as claimed.

We now return to our given composite
\[ \Delta^1 \xra{\ttf} \ttE \stackrel{\ttp}{\fibn} \ttB \]
in $s\Set_\Joyal$.  This can be considered as defining a fibration in $(s\Set_\Joyal)_{\Delta^1/}$, and hence applying our right Quillen functor
\[ (s\Set_\Joyal)_{\Delta^1/} \xra{\beta} \Fun([1],s\Set_\Joyal)_\Reedy \]
yields another fibration.  In particular, the resulting relative matching map
\[ \ttE_{/\ttf} \ra \ttE_{/\ttf(1)} \underset{\ttB_{/\ttp (\ttf(1))}}{\times} \ttB_{/\ttp(\ttf)} \]
at the object $1 \in [1]$ must lie in $\bF_\Joyal \subset s\Set$, as desired.
\end{proof}

Using \cref{invce of cart morphism}, we now show that our model-independent notion of a cartesian fibration is suitably compatible with the quasicategorical notion of a JL-cartesian fibration.

\begin{cor}\label{invce of cart fibn}
Let $\ttE \stackrel{\ttp}{\fibn} \ttB$ be a fibration between fibrant objects in $s\Set_\Joyal$ which presents a map $\E \xra{p} \B$ in $\Cati$.  If $p$ is a cartesian fibration, then $\ttp$ is a JL-cartesian fibration.
\end{cor}

\begin{proof}
Suppose we are given any edge $\ttf \in \ttB_1$.  Let us write $\delta_0(\ttf) = \ttb_2 \in \ttB_0$ and $\delta_1(\ttf) = \ttb_1 \in \ttB_0$.  Suppose further that we are given any vertex $\tte_2 \in \ttE_0$ with $\ttp(\tte_2) = \ttb_2$.  Then we must find a JL-$\ttp$-cartesian lift of the edge $\ttf \in \ttB_1$ relative to the vertex $\tte_2 \in \ttE_0$.

Let us respectively write $e_1 \xra{\tilde{f}} e_2$ and $b_1 \xra{f} b_2$ for the morphisms in $\E$ and $\B$ presented by the maps $\Delta^1 \xra{\tilde{\ttf}} \ttE$ and $\Delta^1 \xra{\ttf} \ttB$ in $s\Set_\Joyal$ (for any choice of edge $\tilde{\ttf} \in \ttE_1$).  Then, according to \cref{invce of cart morphism}, for an edge $\tilde{\ttf} \in \ttE_1$ to be $\ttp$-cartesian, it suffices to verify that the morphism $\tilde{f}$ in $\E$ is $p$-cartesian.

Now, the given data define a vertex
\[ ( \ttf , \tte_2 ) \in \left( \enrhom(\Delta^1,\ttB) \underset{\ev_1 , \ttB , \ttp}{\times} \ttE \right)_0 \]
of the fiber product in $s\Set$.  Moreover, by \cite[Proposition 1.2.7.3(1)]{LurieHTT} and the Reedy trick, this fiber product is a homotopy pullback in $s\Set_\Joyal$.  Hence, this vertex defines an object
\[ ( f , b_2) \in \Fun([1],\B) \underset{t,\B,p}{\times} \E \]
of the pullback in $\Cati$.

Next, it is easy to check that the map
\[ \enrhom(\Delta^1,\ttE) \xra{\enrhom(\Delta^1,\ttp)} \enrhom(\Delta^1 , \ttB) \]
lies in $\bF_\Joyal \subset s\Set$, simply by using
\begin{itemize}
\item the fact that $\bF_\Joyal = \rlp( ( \bW \cap \bC)_\Joyal)$,
\item the adjunction $- \times \Delta^1 : s\Set \adjarr s\Set : \enrhom_{s\Set}(\Delta^1,-)$, and
\item the fact that $s\Set_\Joyal \xra{- \times \Delta^1} s\Set_\Joyal$ is a left Quillen functor and hence in particular preserves acyclic cofibrations.
\end{itemize}
Moreover, the map
\[ \enrhom(\Delta^1,\ttE) \xra{\ev_1} \ttE \]
also lies in $\bF_\Joyal \subset s\Set$, since by applying the dual of \cite[Corollary 2.4.7.12]{LurieHTT} to the map $\ttE \xra{\id_{\ttE}} \ttE$ in $s\Set^f_\Joyal$ we see that it is a JL-cocartesian fibration, which by \cite[Remark 2.0.0.5]{LurieHTT} implies that it is in particular a fibration in $s\Set_\Joyal$.  Thus, our conditions that $\delta_0(\tilde{\ttf}) = \tte_2$ and that $\ttp(\tilde{\ttf}) = \ttf$ translate into the single condition that $\tilde{\ttf} \in \ttE_1$ define a vertex of the limit
\[ \lim \left( \begin{tikzcd}
& & \Delta^0 \arrow{d}{\tte_2} \\
& \enrhom(\Delta^1,\ttE) \arrow[two heads]{r}[swap]{\ev_1} \arrow[two heads]{d}{\enrhom(\Delta^1,\ttp)} & \ttE \\
\Delta^0 \arrow{r}[swap]{\ttf} & \enrhom(\Delta^1,\ttB)
\end{tikzcd} \right) \]
in $s\Set_\Joyal$.

Now, we claim that the above limit is in fact a homotopy limit in $s\Set_\Joyal$.  For this, we appeal to a more elaborate version of the Reedy trick: we endow the category
\[ \Nerve^{-1}(\sd^2(\Delta^1)) = ( \bullet \ra \bullet \la \bullet \ra \bullet \la \bullet) \in \Cat \]
with the Reedy structure determined by the degree function described by the picture $(0 \ra 1 \la 2 \ra 1 \la 0)$.  Using \cite[Proposition 15.10.2(1)]{Hirsch}, it is easy to see that this Reedy category has cofibrant constants, and hence by \cite[Theorem 15.10.8(1)]{Hirsch} we obtain a Quillen adjunction
\[ \const : s\Set_\Joyal \adjarr \Fun(\Nerve^{-1}(\sd^2(\Delta^1)),s\Set_\Joyal)_\Reedy : \lim . \]
Then, to see that the diagram in the above limit defines a fibrant object of $\Fun(\Nerve^{-1}(\sd^2(\Delta^1)),s\Set_\Joyal)_\Reedy$, since it is objectwise fibrant, it only remains to check that the latching map
\[ \enrhom(\Delta^1,\ttE) \xra{(\ev_1,\enrhom(\Delta^1,\ttp)} \ttE \times \enrhom(\Delta^1,\ttB) \]
lies in $\bF_\Joyal \subset s\Set$.  For this, we use the characterization $\bF_\Joyal = \rlp( ( \bW \cap \bC)_\Joyal)$: given any solid commutative square
\[ \begin{tikzcd}
\ttY \arrow{r} \arrow[tail]{d}[swap]{\ttq}[sloped, anchor=south]{\approx} & \enrhom(\Delta^1,\ttE) \arrow{d}{(\ev_1,\enrhom(\Delta^1,\ttp)} \\
\ttZ \arrow{r} \arrow[dashed]{ru} & \ttE \times \enrhom(\Delta^1,\ttB)
\end{tikzcd} \]
in $s\Set_\Joyal$, unwinding the definitions we see that a dotted lift is equivalent to a dotted lift in the diagram
\[ \begin{tikzcd}
(\ttY \times \Delta^1) \underset{\Delta^{\{1\}} , \ttY , \ttq}{\coprod} \ttZ \arrow{r} \arrow{d} & \ttE \arrow{d}{\ttp} \\
\ttZ \times \Delta^1 \arrow{r} \arrow[dashed]{ru} & \ttB
\end{tikzcd} \]
in $s\Set$.  But it is easy to see that the left map lies in $(\bW \cap \bC)_\Joyal$, while the right map lies in $\bF_\Joyal$ by assumption.  Hence the diagram in the above limit defines a fibrant object of $\Fun(\Nerve^{-1}(\sd^2(\Delta^1)),s\Set_\Joyal)_\Reedy$, and so the above limit is a homotopy limit in $s\Set_\Joyal$ and thus presents the limit
\[ \lim \left( \begin{tikzcd}
& & \{ e_2 \} \\
& \Fun([1],\E) \arrow{r}[swap]{t} \arrow{d}{\Fun([1],p)} & \E \arrow[hookleftarrow]{u} \\
\{ f \} \arrow[hook]{r} & \Fun([1],\B)
\end{tikzcd} \right) \]
in $\Cati$.

Now, we have assumed that $p$ is a cartesian fibration, so in particular there must exist a $p$-cartesian lift $\tilde{f}$ of $f$ relative to the object $e_2$.  This defines an object of the above limit in $\Cati$, which must therefore be represented by a vertex of the above homotopy limit in $s\Set_\Joyal$: this selects the desired JL-$\ttp$-cartesian lift $\tilde{\ttf} \in \ttE_1$ of $\ttf \in \ttB_1$ relative to $\tte_2 \in \ttE_0$.  So the map $\ttE \stackrel{\ttp}{\fibn} \ttB$ in $s\Set_\Joyal$ is indeed a JL-cartesian fibration, as claimed.
\end{proof}

\bibliographystyle{amsalpha}
\bibliography{grjl}{}

\end{document}